\documentclass{birkjour}
\usepackage[T1]{fontenc}
\usepackage{amsrefs}
 \numberwithin{equation}{section}
\newtheorem{teo}{Theorem}[section]

\newtheorem{defin}[teo]{Definition}
\newtheorem{remark}[teo]{Remark}

\newtheorem{lemma}[teo]{Lemma}

\newcommand\N{\mathbb{N}}

\renewcommand{\O }{\Omega }
\newcommand\R{\mathbb{R}}
\newcommand\ds{\displaystyle}

\def\elle#1{L^{#1}(\mathcal{Q})}

\def\div{{\rm div}}
\def\elle#1{L^{#1}(\Omega)}
\def\w#1#2{W^{#1,#2}_0(\Omega)}
\def\io{\int_{\mathcal{Q}}}
\def\norma#1#2{\|#1\|_{\lower 4pt \hbox{$\scriptstyle #2$}}}
\def\un{u_n}

\def\finedim
\def\gw{G_{\tilde{k}}(w_n)}

\def\R{I \!\!R}
\def\N{I \!\! N}
\def\elle#1{L^{#1}(\mathcal{Q})}

\def\be{\begin{equation}}
\def\ee{\end{equation}}

\title[Parabolic problems]{Parabolic problems with slightly superlinear convection terms}
\author{Fessel Achhoud}
\address[F.\ Achhoud]{University of Catania, Viale Adrea Doria 6, 95125 Catania, Italy}
\email{fessel.achhoud@studenti.unime.it}

\date{\today}

\begin{document}

\maketitle 
 \begin{abstract}
 In this paper we deal with a non-linear parabolic problem which involving a convection term with  super--linear growth, whose
model is
\[
\frac{\partial u}{\partial t}-\div(\mathcal{M}(x,t)\nabla u)= -\div(u\log (e+|u|)E(x,t))+f(x,t),
\]
where $\mathcal{M}$ is a bounded measurable matrix, the vector field $E$ and the function $f$ belong to suitable Lebesgue spaces.  We prove the existence of a unique bounded and unbounded weak solution.
\end{abstract}

\tableofcontents

\section{Introduction and main results}
The study of parabolic partial differential equations with non-linear convection terms has a rich history, deeply rooted in the analysis of existence, regularity, and qualitative behavior of solutions. Seminal works, such as the foundational paper by Lucio Boccardo et al \cite{bop}, established  results for equations whose simplest model is
$$
u_t - \text{div}(a(x,t,u)\nabla u) = -\text{div}(uE)  \text{ in } \Omega \times (0,T),\quad u(0,x) = u_0(x)  \text{ in } \Omega,
$$
where $\Omega$ is a bounded domain, $T>0,$ $E \in (L^2(\Omega \times (0,T)))^N,$ $u_0(x)\in L^1(\Omega)$ and $a(x,t,u)$ is a bounded  Carathéodory function. Recently, by considering the same equation, the authors in \cite{bopbis} have proved the existence and boundedness of solutions as well as an improvement on the regularity results by considering the following assumptions on the vector field $E$:
\begin{equation*}
\begin{aligned}
&|E(x, t)| \leq \mu|B(x)|, \quad \mu>0, \quad B \in L^N(\Omega) \quad \forall(x, t) \in \Omega\times (0,T),\\
&|E(x, t)| \leq \frac{B}{|x|}, \quad B>0.
\end{aligned}
\end{equation*} 
In recent years, considerable attention has been given to the study of convection terms of the type 
$-\div(h(u)E(x,t))$, aiming to understand how the specific nonlinear structure of the function $h$ affects the existence, uniqueness, and regularity properties of solutions. The interaction between the nonlinearity in $h$ and the external field $E(x,t)$ plays a crucial role in determining the qualitative behavior of the solutions. Especially when $h$ exhibits super-linear growth at infinity, that is,
\[
\lim_{|l| \to \infty} \frac{|h(l)|}{|l|} = +\infty,
\]  
such a property  fundamentally breaks down the coercivity and integrability properties of the associated operators. The nature of this growth, whether polynomial or logarithmic plays a decisive role in the proof of the existence and regularity of solutions. 

In what follows  we denote by \(\mathcal{Q} \) the cylinder \( \Omega \times (0, T) \) and $\Sigma:=\partial \Omega\times(0,T)$. We shall address the following nonlinear parabolic Cauchy–Dirichlet problem
\begin{equation} \label{problem}
\begin{cases}
\ds\frac{\partial u}{\partial t} -\div(\mathcal{M}(x,t)\nabla u)= -\div(h(u)E(x,t))+f(x,t) \qquad & \mbox{in } \mathcal{Q},\\
u (x,0) = 0 & \mbox{in }  \Omega,\\
u=0 & \mbox{on } \Sigma.
\end{cases}
\end{equation}
Here, $M: \mathcal{Q}\rightarrow \R^N$ is a measurable matrix such that
\begin{equation}\label{alfa}
 \alpha |\xi|^2\leq 
  \mathcal{M}(x,t)\xi\cdot\xi \leq\beta|\xi|^2,  \quad  \mbox{a.e.}\; (x,t)\in\mathcal{Q},\quad \forall\; \;\xi\in\R^N,
\end{equation}
for some constants $\alpha, \beta > 0$. The vector field $E$ and the source term $f$ are assumed to belong to appropriate Lebesgue spaces, while the nonlinearity $h: \mathbb{R} \to \mathbb{R}$ is given by
\begin{equation} \label{ipoh}
h(l) := l \log(e + |l|), \quad \text{for all } l \in \mathbb{R}.
\end{equation}

In the stationary case, with power-type super-linearity functions like \( u|u|^\theta \) (\(\theta \geq 0\)) which grow polynomially,  was studied in the sixties by Stampacchia (see \cite{stamp}) with restrictive smallness conditions on the data \( E \) or \( f \).  For more and different aspects concerning a very weak hypothesis on the vector field \( E \) we refer to \cite{Bumi2009, Bumi2012, jde, cirmi2022}.

 The authors of \cite{bbc2024, marah} recently analyzed equations with slightly super-linear growth. The logarithmic factor ensures that \( h\) retains sufficient integrability properties even for large \(u \). This delicate balance allows solutions to exist without imposing size restrictions on \( E \) or \( f \). Unlike power-law terms, which dominate the diffusion operator \( -\operatorname{div}(\mathcal{M}(x,t)\nabla u) \) at infinity, the logarithmic growth moderates this competition.

 First, we prove the existence of bounded weak solutions under optimal integrability conditions on \(E \) and \(f\), extending the elliptic existence results to the parabolic case. 
\begin{teo}\label{generalh}
Assume \eqref{alfa}, \eqref{ipoh} and take $E\in [\elle r]^N$,  with $ r>N+2$, and $f\in \elle m$, with $ m>\ds\frac{N}{2}+1$. Then, there exists a unique weak solution $u\in L^2(0,T; W^{1,2}_0(\Omega))\cap \elle{\infty}$ to problem \eqref{problem}.
\end{teo}

Second, we are also able to provide existence of unbounded solutions, as the following Theorem states.
\begin{teo}
\label{logarithmich}
Assume \eqref{alfa}, take
$E\in (L^r(\mathcal{Q}))^N$,  with $ r>N+2$, and 
$f\in  L^{2}(0,T; W^{-1,2}(\Omega))$. Then, there exists a unique weak solution $u\in L^{2}(0,T; W^{1,2}_0(\Omega))$  to problem \eqref{problem}.
\end{teo}
%
\begin{remark}
We note that we set all our problems in $L^{2}(0,T; W^{1,2}_0(\Omega))$. However the results still 
hold, with the same proofs, for differential operators with growth of order $p>1$ in the 
$L^{p}(0,T; W^{1,p}_0(\Omega))$ framework.
\end{remark}
\begin{remark}
We point out that the results above still hold if we consider general nonlinear convection term of the form $div(\mathcal{F}(x,t,u))$ such that
\begin{equation*}
\vert \mathcal{F}(x,t,\eta)\vert\leq E(x,t)\eta\vert \eta\vert^{\theta-1} \log^\beta(e+\vert \eta\vert),\quad 0\leq \theta,\beta\leq 1.
\end{equation*}

\end{remark}
 
The plan of the paper is as follows: In Sect. 2, we will recall and prove some auxiliary results, while in Sect. 3, we will give the proof of our main  result for the existence of bounded and unbounded solution to the problem \eqref{problem}.

\section{Preliminary results}
In this section we consider a general real function  $h$ which satisfies

\begin{equation}\label{gehy}
h\in W^{1,\infty}_{loc}(\Omega),\quad h(0)=0,\quad\text{ and }\quad
\lim_{|l| \to \infty} \frac{|h(l)|}{|l|} = +\infty,
\end{equation}
 Let us define the following functions:
\be\label{nonlintest}
\varphi(r)=\int_0^r\frac{dl}{(|h(l)|+1)^2},
\ee
and 
\begin{equation}
\label{defH}
H(r)=\int_0^r\frac{dl}{|h(l)|+1},
\end{equation}
such that
\be\label{growthg}
\lim_{|l|\to \infty}|H(l)|=\infty.
\ee

Given that uniqueness constitutes a central objective of this work, we highlight that comparison principle between sub- and super-solutions serve as the foundational tool for establishing such result. To rigorously develop this framework, we introduce the following definitions:
\begin{defin}
 We say that $v \in L^2\left(0, T ; W^{1,2}_0(\Omega)\right)$ is a subsolution of problem \eqref{problem} if $h(v)E(x,t) \in$ $(L^2(0,T;W^{-1,2}(\Omega)))^N$ and

\begin{equation}\label{inq1C}
\begin{aligned}
&\int_0^T \int_\Omega \frac{\partial v}{\partial t}\phi\;dxdt+\int_0^T \int_\Omega \mathcal{M}(x,t)\nabla v \nabla \phi\;dxdt 
\\&\le \int_0^T \int_\Omega h(v) E(x,t)\nabla \phi \;dxdt+\int_0^T \int_\Omega f\phi \;dxdt,
\end{aligned}
\end{equation}
for any $0\le \phi\in L^2(0,T;\w12)\cap \elle{\infty}.$
\end{defin}
\begin{defin}
We say that $w \in L^2\left(0, T ; W^{1,2}_0(\Omega)\right)$ is a supersolution of problem \eqref{problem} if $h(w)E(x,t) \in$ $(L^2(0,T;W^{-1,2}(\Omega)))^N$ and

\begin{equation}\label{inq2C}
\begin{aligned}
&\int_0^T \int_\Omega \frac{\partial w}{\partial t}\phi\;dxdt+\int_0^T \int_\Omega \mathcal{M}(x,t)\nabla w \nabla \phi\;dxdt 
\\&\ge \int_0^T \int_\Omega h(w) E(x,t)\nabla \phi \;dxdt+\int_0^T \int_\Omega f\phi \;dxdt,
\end{aligned}
\end{equation}
for any $0\le \phi\in L^2(0,T;\w12)\cap \elle{\infty}$.

\end{defin}
Now we are able to state and prove our main comparison lemma that will play the key role
in the proof of our main results. 

\begin{lemma}\label{comparison}
Assume that the  assumptions \eqref{alfa} and \eqref{gehy} hold. Let  $ f\in L^2(0,T;W^{-1,2}(\Omega))$, $E\in [\elle 2]^N$ and consider two functions $$v,w\in L^2(0, T ; W^{1,2}_0(\Omega)) \cap L^{\infty}(Q)$$ such that are respectively a subsolution and a supersolution of the problem \eqref{problem}.
 Therefore, we have 
 $$v\le w \quad  \text{ a.e. in }\Omega.$$
Moreover, if $v,w \in L^2(0, T ; W^{1,2}_0(\Omega))$ and we additionally assume that
\[
|h'(l)|\le c(|l|^{\theta}+1)\,  \mbox{ a.e. in } \mathbb R  \mbox{ and }  (|v|^{2\theta}+|w|^{2\theta})|E|^2\in\elle 1\quad \forall \theta>0,
\]
we can again conclude that 
$$v\le w \quad  \text{ a.e. in }\Omega.$$
\end{lemma}
\begin{proof}
Let us take $T_{{\epsilon}}(v-w)_+$ as a test function in the inequalities \eqref{inq1C} and \eqref{inq2C}; thus, taking the difference we get
\[
\begin{split}
&\int_0^\tau \int_\Omega \frac{\partial(v-w)_+}{\partial t} T_\varepsilon(v-w)_+\;dxdt+\alpha\io |\nabla T_{{\epsilon}}(v-w)_+ |^2\;dxdt\\&\le \int_0^\tau\int_{\{0<v-w<{\epsilon}\}} \left(h(v)-h(w)\right) E(x,t)\nabla T_{{\epsilon}}(v-w)_+\;dxdt\\
\le& \frac{1}{2\alpha}\int_0^\tau\int_{\{0< v-w< {\epsilon}\}}\left|h(v)-h(w)\right|^2 |E(x,t)|^2\;dxdt+\frac{\alpha}2\io |\nabla T_{{\epsilon}}(v-w)_+ |^2\;dxdt.
\end{split}
\]
Setting $z:=(v-w)_+$ , we deduce
\[
\begin{split}
&\frac{1}2\int_\Omega z^2(\tau)\;dx+\frac{\alpha}2\io |\nabla T_{{\epsilon}}(z) |^2\;dxdt\le \frac{1}{2\alpha}\int_{\{0< z< {\epsilon}\}}\left|h(v)-h(w)\right|^2 |E(x,t)|^2\;dxdt.
\end{split}
\]

We shall prove that $(v-w)_+\equiv 0$ a.e. in $\Omega$. To this aim, we need to show that 
\begin{equation}\label{eqlm}
\lim_{\epsilon\rightarrow 0}\int_0^\tau\int_{\{0< z< {\epsilon}\}}\left|h(v)-h(w)\right|^2 |E(x,t)|^2\;dxdt=0.
\end{equation}
In order to do it we discuss to cases:

If $v,w\in\elle{\infty}$, let us set $|v|+|w|\le M$ and use the assumption $h\in W^{1,\infty}_{loc}(\mathbb{R})$ to deduce that
\[
\int_0^\tau\int_{\{0< z< {\epsilon}\}}\left|h(v)-h(w)\right|^2 |E(x,t)|^2\;dxdt\le C L_M^2 \int_{\{0< z< {\epsilon}\}}|E(x)|^2 \;dxdt,
\]
for a suitable $L_M$ independent on $\epsilon$. So, since $E\in [\elle 2]^N$, \eqref{eqlm} holds.

On the other hand, if $|h'(l)|\le c(|l|^{\theta}+1)$, we have that
\[
|h(v)-h(w)|\le c|v-w|\int_0^1\left(|w+t(v-w)|^{\theta}+1\right)dt\le \tilde c|v-w|(|v|^{\theta}+|w|^{\theta}+1 ),
\]
thus, we get 
\begin{align*}
 &\int_0^\tau\int_{\{0< z< {\epsilon}\}}\left|h(v)-h(w)\right|^2 |E(x,t)|^2\;dxdt\\&\le C\int_0^\tau\int_{\{0< z< {\epsilon}\}}(|v|^{2\theta}+|w|^{2\theta}+1 )|E(x,t)|^2\;dxdt,
\end{align*}
which, since $\left(|v|^{2\theta}+|w|^{2\theta}\right)|E|^2\in\elle 1,$ implies that \eqref{eqlm} holds true.
\end{proof}

 We recall the fundamental inequality of Gagliardo-Niremberg.

\begin{lemma}(Gagliardo-Nirenberg inequality)\label{lm1}
States that if \( v \) is a function in \( W_0^{1,q}(\Omega) \cap L^\rho(\Omega) \), with \( q \geq 1 \) and \( \rho \geq 1 \), then there exists a positive constant \(\tilde{C} \), dependent on \( N \), \( q \), and \( \rho \), such that:
\begin{equation}\label{inq1}
\|v\|_{L^z(\Omega)} \leq \tilde{C}\|\nabla v\|_{\left(L^q(\Omega)\right)^N}^\theta\|v\|_{L^\rho(\Omega)}^{1-\theta},
\end{equation}
for every \( \theta \) and \( \gamma \) satisfying:
\begin{equation}
0 \leq \theta \leq 1, \quad 1 \leq z<+\infty, \quad \frac{1}{z}=\theta\left(\frac{1}{q}-\frac{1}{N}\right)+\frac{1-\theta}{\rho}.
\end{equation}
\end{lemma}
\begin{proof}
See \cite{DiBenedetto_1993}.
\end{proof}
 An immediate consequence of the previous lemma is the following embedding result:
\begin{lemma}\label{lm2}
 Suppose \( v \in L^q(0, T ; W_0^{1, q}(\Omega)) \cap L^{\infty}\left(0, T ; L^{\rho}(\Omega)\right) \), where \( \rho, q \geq 1 \). Then \( v \) is also in \( L^{\varrho}(Q) \), where \( \varrho = q \frac{(N+\rho)}{N} \). Moreover, there exists a positive constant \(\bar{C} \), dependent only on \( N, q, \rho \), such that:
 \begin{equation}\label{inq2}
 \int_{\mathcal{Q}}|v|^{\varrho}  dx dt \leq \bar{C}\left(\operatorname{ess} \sup _{0\leq t\leq T} \int_{\Omega}|v|^{\rho} \, dx\right)^{\frac{q}{N}} \int_{\mathcal{Q}}|\nabla v|^q  dx dt.
 \end{equation}
\end{lemma}
%
%
\begin{proof}
See  Proposition 3.1 in \cite{DiBenedetto_1993}.
\end{proof}

We prove now a Lemma, that will play an essential role in the proof of our main results.
\begin{lemma}\label{lme}
We have, for any $l\in \R$, that
\begin{equation}
\Phi(l)\geq H^2(l).
\end{equation}
with $\Phi(l)=\ds\int_0^l \varphi(s)\;ds.$
\end{lemma}
\begin{proof}
Let $t>0$, we have
\begin{equation}
H(s)=\int_0^s \frac{dl}{1+h(l)}\leq s^{\frac{1}{2}}\left( \int_0^s \frac{dl}{(1+h(l))^2}\right)^2
\end{equation}
Otherwise, by integration by part, we write
\begin{align*}
\Phi(s)&=\int_0^s \varphi(l)dl=\int_0^s\left( \int_0^l\frac{dr}{(|h(r)|+1)^2}\right)\cdot 1 ds,\\
&= \left[l\varphi(l) \right]^s_0-\int_0^s l\varphi^\prime(l)dl,\\
&\geq H^2(s)-\int_0^s \frac{l dl}{(1+h(l))^2}.
\end{align*}
Moreover, we have
\begin{align*}
H(s)=\left[\frac{l}{1+h(l)} \right]_0^s+\int_0^s \frac{l h^\prime(l) dl}{(1+h(l))^2}.
\end{align*}
If $h^\prime(l)>\alpha$ with $\alpha>0$, then 
$$\int_0^s \frac{l  dl}{(1+h(l))^2}\geq \frac{1}{\alpha} \left[H(s)-\frac{s}{1+h(s)} \right].$$
Thus, by the previous inequality we get
$$\Phi(s)\geq H^2(s)-\frac{1}{\alpha}H(s)+\frac{s}{1+h(s)}.$$
\end{proof} 
 
\section{Proof of the existence results}
In order to prove our existence and regularity results, we will consider the approximating
problems
\begin{equation} \label{prapr}
\begin{cases}
\ds\frac{\partial u_n}{\partial t} -\div(\mathcal{M}(x,t)\nabla u_n)= -\div(h(T_n(u_n))E_n(x,t))+f_n(x,t) \qquad & \mbox{in } \mathcal{Q},\\
u_n (x,0) = 0 & \mbox{in }  \Omega,\\
u_n=0 & \mbox{on } \Sigma.
\end{cases}
\end{equation}
By known results (see, for instance, \cite{LL}), there exists at least a weak solution $u_n$ of \eqref{prapr} which
belongs to $L^2(0,T; W^{1,2}_0(\Omega))\cap C^0([0,T];L^2(\Omega))$ and satisfies 
\begin{equation}\label{wf}
\begin{cases}
\ds h(u_n)E_n(x,t)\in (L^2(0,T; W^{-1,2}(\Omega)))^N\\
\\
\ds\int_0^T\int_\Omega \frac{\partial u_n}{\partial t}\phi \;dxdt+\int_0^T\int_\Omega \mathcal{M}(x,t)\nabla u_n\nabla \phi\;dxdt\\=\ds\int_0^T\int_\Omega h(u_n)E_n(x,t)\nabla \varphi\;dxdt+\int_0^T\int_\Omega f_n(x,t)\phi\;dxdt
\end{cases}
\end{equation}
for every $\phi \in L^2(0,T; W^{1,2}_0(\Omega)).$

If   the assumption \eqref{ipoh} holds, there exists $C_1>0$ such that $|\varphi(s)| \leq C_1$ for any  $n \in \mathbb{N}$. The main tool of that paper is the following a priori decay for the 
measure of superlevel sets of solutions
\begin{lemma} \label{decay}
Assume \eqref{alfa}, \eqref{gehy}, and take $\mu\ge0$, $E \in (L^2(\mathcal{Q}))^N$, and $f \in L^1(\mathcal{Q})$. Then there exists $C=C(\alpha,E,f)>0$ such that
\begin{equation}
\label{stima_decay}
|\{|\un|>k\}|\le \frac{C}{|H(k)|^{\frac{2(N+2)}{N}}} \quad  \forall \  n \in \mathbb{N},\, k\in \mathbb{R}
\end{equation}
\end{lemma}
\begin{proof}
Given $k>0$ we consider the function $\phi=G_{\varphi(k)}(\varphi(\un^+))\chi_{(0,\tau)}$ as test function in \eqref{wf}
\begin{equation*}
\begin{aligned}
\ds&\int_0^\tau\int_\Omega \frac{\partial u_n}{\partial t}G_{\varphi(k)}(\varphi(\un^+)) \;dxdt+\int_0^\tau\int_\Omega \mathcal{M}(x,t)\nabla u_n\nabla G_{\varphi(k)}(\varphi(\un^+))\;dxdt\\&=\ds\int_0^\tau\int_\Omega h(u_n)E_n(x,t)\nabla G_{\varphi(k)}(\varphi(\un^+))\;dxdt+\int_0^\tau\int_\Omega f(x,t)G_{\varphi(k)}(\varphi(\un^+))\;dxdt
\end{aligned}
\end{equation*}
Observing that by construction, we have 
\be\label{setset}
\begin{split}
\left\{(x,t) \in \mathcal{Q} \, : \varphi(\un(x,t)) > \varphi(k)  \right\}=\left\{(x,t) \in \mathcal{Q} \, : \un(x,t) >k \right\}
\end{split}
\ee 
which implies that
\[
\begin{split}
&\int_0^\tau\int_{\{k<\un^+\}} \frac{\partial u_n}{\partial t}\varphi(\un^+)\;dxdt+
\alpha\int_0^\tau\int_{\{k<\un^+\}} \frac{|\nabla \un^+|^2}{(|h(\un^+)|+1)^2}\;dxdt\\&\le  \int_0^\tau\int_{\{k<\un^+\}} |h(\un^+)||E| \frac{|\nabla \un^+|}{(|h(\un^+)|+1)^2}\;dxdt+\int_0^\tau\int_{\{k<\un^+\}} |f| \varphi(\un^+)\;dxdt\\
\end{split}
\] 
since we have $\ds\frac{|h(\un^+)}{|h(\un^+)|+1}<1$ and $|\varphi(\un^+)|\le C_1$, then we get
\[
\begin{split}
&\int_{\{k<\un^+\}} \Phi(u_n^+(x,\tau))\;dx+
\alpha\int_0^\tau\int_{\{k<\un^+\}} \frac{|\nabla \un^+|^2}{(|h(\un^+)|+1)^2}\;dxdt\\&
\le  \int_0^\tau\int_{\{k<\un^+\}} |E| \frac{|\nabla \un^+|}{|h(\un^+)|+1}\;dxdt+C_1 \int_0^\tau\int_{\{k<\un^+\}}|f(x)|\;dxdt,
\end{split}
\] 
with $\Phi(s)=\ds\int_0^s \varphi(r)\;dr.$

Applying Young's inequality and recalling property \eqref{setset} we obtain that

\begin{equation*}
\begin{split}
&\int_{\{k<\un^+\}} \Phi(u_n^+(x,\tau))\;dx +
\alpha\int_0^\tau\int_{\{k<\un^+\}} \vert\nabla H(u_n^+)\vert^2\;dxdt \\&\le  \frac{\alpha}{2}\int_0^\tau\int_{\{k<\un^+\}} |\nabla H(\un^+)|^2\;dxdt+\frac{1}{2\alpha}\int_0^\tau\int_{\{k<\un^+\}} |E(x)|^2\;dxdt\\&+ C_1 \int_0^\tau\int_{\{k<\un^+\}}|f(x)|\;dxdt.
\end{split}
\end{equation*}
Therefore, we get
\begin{equation*}
\begin{split}
&\int_{\{k<\un^+\}} \Phi(u_n^+(x,\tau))\;dx +
\frac{\alpha}{2}\int_0^\tau\int_{\{k<\un^+\}} \vert\nabla H(u_n^+)\vert^2\;dxdt \\&\le  \frac{1}{2\alpha}\int_0^\tau\int_{\{k<\un^+\}} |E(x)|^2\;dxdt+ C_1 \int_0^\tau\int_{\{k<\un^+\}}|f(x)|\;dxdt.
\end{split}
\end{equation*}
Thanks to the Lemma \ref{lme}, we have
\begin{equation*}
\begin{split}
&\int_{\{k<\un^+\}} H^2(u_n^+(x,\tau))\;dx +
\frac{\alpha}{2}\int_0^\tau\int_{\{k<\un^+\}} \vert\nabla H(u_n^+)\vert^2\;dxdt \\&\le  \frac{1}{2\alpha}\int_0^\tau\int_{\{k<\un^+\}} |E(x)|^2\;dxdt+ C_1 \int_0^\tau\int_{\{k<\un^+\}}|f(x)|\;dxdt.
\end{split}
\end{equation*}
Taking the supremum on $\tau$ in $(0,T)$, we get
\begin{equation}\label{inqb}
\ds \sup_{\tau\in [0,T]}\int_{\{k<\un^+\}} H^2(u_n^+(x,T))\;dx +
\frac{\alpha}{2}\int_0^T\int_{\{k<\un^+\}} \vert\nabla H(u_n^+)\vert^2\;dxdt\leq C(\alpha,E,f),
\end{equation}
with $$C(\alpha,E,f):=\ds\frac{1}{2\alpha}\int_0^T\int_{\{k<\un^+\}} |E(x)|^2\;dxdt+ C_1 \int_0^T\int_{\{k<\un^+\}}|f(x)|\;dxdt$$
Thanks to the Lemma \ref{lm2}, we obtain that
{\footnotesize
\begin{align*}
\ds &\int_0^T\int_{\{k<\un^+\}} \vert H(u_n^+)\vert^{\frac{2(N+2)}{N}} \;dxdt \\&\leq \bar{C}\left(\sup_{\tau\in [0,T]}\int_{\{k<\un^+\}} H^2(u_n^+(x,T))\;dx \right)^{\frac{2}{N}}\int_0^T\int_{\{k<\un^+\}} \vert\nabla H(u_n^+)\vert^2\;dxdt.
\end{align*}}

Since we have
$$
\left\{(x,t) \in \mathcal{Q} \, :  \un^+(x,t) \vert>k \right\}
=\left\{(x,t) \in \mathcal{Q} \, :    H(\un^+(x,t)) > H(k)  \right\},
$$
thus, we derive that
\begin{equation*}
\vert \{k<\un^+\}\vert\leq \frac{C(\alpha,E,f)}{(H(k))^{\frac{2(N+2)}{N}}}.
\end{equation*}

 Repeating the same argument used before with the test function $v=G_{\varphi(k)}(\varphi(\un^-))\chi_{(0,\tau)}$, we derive the following estimate
 \begin{equation*}
\vert \{-k>\un^-\}\vert\leq \frac{\bar{C}(\alpha,E,f)}{(H(k))^{\frac{2(N+2)}{N}}}.
\end{equation*}
Hence, by recalling that $\un=\un^++\un^-$, the estimate \eqref{stima_decay} holds.
\end{proof}

%
%
%
%
Let us provide now the proof of existence of bounded solutions.
\begin{proof}[Proof of Theorem \ref{generalh}]

Defining $\eta=H(k)$ and $z_n=H(\un)$, the inequality \eqref{inqb} gives
\begin{equation}\label{inqb1}
\ds \sup_{\tau\in [0,T]}\int_{\Omega} \vert G_\eta(z_n))\vert^2\;dx +
\frac{\alpha}{2}\int_0^T\int_{\Omega} \vert\nabla G_\eta(z_n)\vert^2\;dxdt\leq C(\alpha,E,f),
\end{equation}

Since $|E|^2+|f|$ belongs to $\elle m$ with $m>\frac{N}{2}$, it follows from a result due to D.G. Aronson and J. Serrin (see \cite{Aronson_1967},) that there exists a positive constant $\mathcal{C}$ such that
\[
\| z_n\|_{\elle{\infty}}=\|H(\un)\|_{\elle{\infty}}\le \mathcal{C},\quad \text{ for every $n\in \N$.}
\]

Thus, the  assumption \eqref{growthg}, implies that
\[
\|\un\|_{\elle{\infty}}\le H^{-1}(C).
\] 
Moreover, estimate \eqref{inqb1} implies  that $$ \|z_n\|_{L^2(0,T,W^{1,2}_0(\Omega))}\le \mathcal{C},$$
 so that  $\{\un\}$ is bounded also in $ L^2(0,T,W^{1,2}_0(\Omega))$.\\
 
 Therefore,  we conclude that, up to subsequence (still denoted by $u_n$) there exists  $u\in L^2(0,T,W^{1,2}_0(\Omega))$ such that 
\begin{equation}\label{cvb1}
u_n \rightharpoonup u \quad \text{ weakly in $L^2(0,T,W^{1,2}_0(\Omega))$}.
\end{equation}
This implies,  by Aubin type lemma (see \cite{Simon}), that there exists a subsequence, still denoted by $u_n$, such that
\begin{equation}\label{cvb2}
u_n \rightarrow u \quad \text{ a.e. in $\mathcal{Q}$}.
\end{equation} 
For any given measurable set $\mathcal{E}\subset \mathcal{Q}$ and by Hölder inequality, we have
\begin{equation*}
\int_{\mathcal{E}} \vert h(u_n)E(x,t)\nabla \varphi\vert\;dxdt\le C\left(\int_\mathcal{E}\vert \nabla \varphi\vert^2\;dxdt, \right)^{\frac{1}{2}}
\end{equation*}
so that, by the convergences \eqref{cvb1}, \eqref{cvb2} and Vitali's theorem, we can pass to the limit in the approximate problem \eqref{wf} and the result is proved. The uniqueness follows from Theorem \ref{comparison}.
\end{proof}
\begin{remark}
Notice that the estimate \eqref{stima_decay} is true for all nonlinearity $h(u)$. However, as we have just seen in the proof of Theorem \ref{generalh}, it is useful only if $H(k)\to\infty$ as $k$ diverges. Another consequence of \eqref{growthg} is that for any $\epsilon >0$ there exists $k_{\epsilon}$ such that $\forall \, k > k_{\epsilon}$ it follows that
\begin{equation}
\label{levelset}
|\{(x,t) \in \mathcal{Q} : \quad |u_n(x,t)|>k\}| \leq \epsilon , \quad \text{uniformly \, w.r. to } n \in \mathbb{N}.
\end{equation}
This estimate is crucial in order to prove the boundedness of $\{\un\}$ in the energy space, at least for some choices of the nonlinearity $h(u)$.
\end{remark}
Let $k>0$, in the sequel, for any measurable function $u$ defined on $\mathcal{Q}$ and any $l>0$, we define 
$$\mathcal{Q}_l:=\{|u|> k\}\times (0,l).$$

\begin{proof}[Proof of Theorem \ref{logarithmich}]
The first step is to show that, for any $a\ge1$, there exists\\ $C=C(a, f, E)$ such that 
\begin{equation}
\label{stimaloga}
\|\log^a(e+|\un|)\|_{\elle{\frac{2(N+2)}{N}}}\le C.
\end{equation}
Let us set
\[
v=\int_0^{\un}\frac{\log^{2(a-1)}(e+|s|)}{(e+|s|)^2}ds,
\] 
and notice that $|v|\le C_a$, for some positive constant $C_a$. Taking $v$ as a test function in \eqref{wf} and using Young's inequality, we get
\[
\begin{split}
&\int_0^\tau \int_\Omega \frac{\partial u_n}{\partial t}v\;dxdt+\alpha\io|\nabla \un|^2\frac{\log^{2(a-1)}(e+|\un|)}{(e+|\un|)^2}\;dxdt \\\le& \io |E(x)||\nabla \un| \frac{\log^{2a-1}(e+|\un|)}{(e+|\un|)^2}\;dxdt +C_a\|f\|_{\elle 1}
\\ \le& \frac{1}{2\alpha}\io\log^{2a}(e+|\un|)|E(x)|^{2}\;dxdt+\frac{\alpha}{2}\io|\nabla \un|^2\frac{\log^{2(a-1)}(e+|\un|)}{(e+|\un|)^2}\;dxdt\\+&C_a\|f\|_{\elle 1}.\\
\end{split}
\]
Let $\ds\Psi(s)=\int_0^s v(r)\;dr,$  we have
\begin{equation*}
\int_0^\tau \int_\Omega \frac{\partial u_n}{\partial t}v\;dxdt=\int_\Omega \Psi(u_n(x,\tau))\;dx,
\end{equation*}
\[
\begin{split}
\frac{\alpha}{2a^2}\io|\nabla \log^a(e+|\un|)|^2\;dxdt&=\frac{\alpha}{2}\io|\nabla \un|^2\frac{\log^{2(a-1)}(e+|\un|)}{(e+|\un|)^2}\;dxdt,
\end{split}
\]
and
\[
\begin{split}
 \frac{1}{2\alpha}\io\log^{2a}(e+|\un|)|E(x)|^{2}\;dxdt&= \frac{1}{2\alpha}\io\left[\log^{a}(e+|\un|)+1-1\right]^2|E(x)|^{2}\;dxdt\\
&\le \frac{1}{\alpha}\int_{\mathcal{Q}_\tau}(\log^{a}(e+|\un|)+1)^2|E(x)|^{2}\;dxdt\\&+\frac{1}{\alpha}\|E\|_{\elle 2}^2+\frac{\log^{2a}(e+k)}{2\alpha}\|E\|_{\elle 2}^2.
\end{split}
\]
Passing to the supremum for  $\tau \in (0,T)$, we get after applying Holder inequality,
{\footnotesize
\begin{align*}
 \sup_{\tau\in(0,\tau)}&\int_\Omega \Psi(u_n(x,\tau))\;dx+\frac{\alpha}{2a^2}\io|\nabla \left[\log^a(e+|\un|)+1\right]|^2\;dxdt \\ \le& \frac{1}{\alpha}\int_{\mathcal{Q}_{\tau}}\left[\log^{a}(e+|\un|)+1\right]^2|E(x)|^{2}\;dxdt+\frac{\log^{2a}(e+k)+2}{2\alpha}\|E\|_{\elle 2}^2+C_a\|f\|_{\elle 1},\\\le&
\left[\io \left[\log^{a}(e+|\un|)+1\right]^{\frac{2(N+2)}{N}}\;dxdt\right]^{\frac{N}{N+2}}\left[\int_{\mathcal{Q}_{\tau}} |E(x,t)|^{N+2}\;dxdt\right]^{\frac{2}{N+2}}\\+&\frac{\log^{2a}(e+k)+2}{2\alpha}\|E\|_{\elle 2}^2+C_a\|f\|_{\elle 1}.
\end{align*}}
Observing that for a positive constant $\mathcal{C}_a<\ds\left(\frac{1}{2a}\right)^2$, we have
$$\Psi(u_n(x,\tau))\geq \mathcal{C}_a\left[\log^a(e+|\un(x,\tau)|)+1 \right]^2.$$
Therefore, we obtain that
\begin{equation*}
\begin{aligned}
&\sup_{\tau\in(0,\tau)}\mathcal{C}_a\int_\Omega \vert w_n(x,\tau)\vert^2\;dx+\frac{\alpha}{2a^2}\io|\nabla w_n|^2\;dxdt  \\&\le 
\left[\io \vert w_n\vert^{\frac{2(N+2)}{N}}\;dxdt\right]^{\frac{N}{N+2}}\Vert E\Vert^2_{L^{N+2}(\mathcal{Q}_{\tau})}\\&+\frac{\log^{2a}(e+k)+2}{2\alpha}\|E\|_{\elle 2}^2+\mathcal{C}_a\|f\|_{\elle 1},
\end{aligned}
\end{equation*}
where $$w_n(x,t):=\log^a(e+|\un(x,t)|)+1.$$
Applying Lemma \ref{lm2} with $v = w_n$, $q=\rho= 2$ and $\varrho=\frac{2(N+2)}{N}$, and thanks to the previous inequality, we get
\[
\begin{split}
\int_{\mathcal{Q}}\vert w_n\vert^{\frac{2(N+2)}{N}}\;dxdt\le \mathcal{L}\left(\int_{\mathcal{Q}}\vert w_n\vert^{\frac{2(N+2)}{N}}\;dxdt \right)\Vert E\Vert^2_{L^{N+2}(\mathcal{Q}_{\tau})}+\mathcal{R},
\end{split}
\]
where $$\mathcal{L}:=2^{\frac{N}{2}}\bar{C}\frac{2a^2}{\alpha \mathcal{C}_a^{\frac{2}{N}}}\text{ and }\mathcal{R}:=\mathcal{L}\left[\frac{\log^{2a}(e+k)+2}{2\alpha}\|E\|_{\elle 2}^2+C_a\|f\|_{\elle 1} \right]^{\frac{2}{N}+1}.$$
Thanks to Lemma \ref{decay} it's possible to choose $k$ large enough  in such a way that
$$\Vert E\Vert^2_{L^{N+2}(\mathcal{Q}_{\tau})}\le \frac{1}{2\mathcal{L}},$$
so that, we deduce 
%
\begin{equation*}
\left\Vert w_n\right\Vert_{L^{\frac{2(N+2)}{N}}(\mathcal{Q})}=\left\Vert \log^{a}(e+|\un|)+1\right\Vert_{L^{\frac{2(N+2)}{N}}(\mathcal{Q})}<\mathcal{C}.
\end{equation*}

Next, we will prove that $\left\{\un\right\}$ is bounded in $\w12$.

Picking up $\phi=\un\chi_{(0,\tau)},$ $\tau\in (0,T)$,  in \eqref{wf}, we get
\begin{equation*}
\begin{aligned}
\ds&\int_0^\tau\int_\Omega \frac{\partial u_n}{\partial t}u_n \;dxdt+\int_0^\tau\int_\Omega \mathcal{M}(x,t)\nabla u_n\nabla u_n\;dxdt\\&=\ds\int_0^\tau\int_\Omega h(u_n)E_n(x,t)\nabla u_n\;dxdt+\int_0^\tau\int_\Omega f(x,t)u_n\;dxdt.
\end{aligned}
\end{equation*}
Using \eqref{alfa}, applying Young's inequality, we obtain
\begin{equation*}
\label{notte}
\begin{split}
&\ds\frac{1}{2}\int_\Omega \vert u_n(\tau)\vert^2 \;dx+\frac{\alpha}{2}\int_0^\tau\int_{\Omega}|\nabla \un|^2\;dxdt\\&\le \frac{1}{2\alpha}\int_0^\tau\int_{\{|\un|> k|\}}|\un|^2\log^2(e+|\un|)|E|^2\;dxdt\\
&+\frac{k^2\log^2(e+k)}{2\alpha}\int_0^\tau\int_{\Omega}|E|^2\;dxdt+\int_0^\tau\int_\Omega |f||\un| \;dxdt.
\end{split}
\end{equation*}
Passing to the supremum for  $\tau \in (0,T)$, we acquire
\begin{equation*}
\label{notte}
\begin{split}
&\sup_{\tau\in (0,\tau)}\frac{1}{2}\int_\Omega \vert u_n(\tau)\vert^2 \;dx+\frac{\alpha}{2}\int_0^{\tau} \|\nabla \un\|^2_{(L^2(\Omega))^N}\;dxdt\\&\le \frac{1}{2\alpha}\int_{\mathcal{Q}_{\tau}}|\un|^2\log^2(e+|\un|)|E|^2\;dxdt\\
&+\frac{k^2\log^2(e+k)}{2\alpha}\|E\|^2_{\elle2}+\int_0^{\tau} \Vert f\Vert_{W^{-1,2}(\Omega)}\Vert u_n\Vert_{W^{1,2}_0(\Omega)}\;dt,
\end{split}
\end{equation*}
Applying Young's and Poincaré's inequalies to get
\begin{equation*}
\label{notte1}
\begin{split}
&\sup_{\tau\in (0,\tau)}\frac{1}{2}\int_\Omega \vert u_n(\tau)\vert^2 \;dx+\frac{\alpha}{4}\int_0^{\tau} \|\nabla \un\|^2_{(L^2(\Omega))^N}\;dt\\&\le \frac{1}{2\alpha}\int_{\mathcal{Q}_{\tau}}|\un|^2\log^2(e+|\un|)|E|^2\;dxdt\\
&+\frac{k^2\log^2(e+k)}{2\alpha}\|E\|^2_{\elle2}+C\int_0^{\tau} \Vert f\Vert^2_{W^{-1,2}(\Omega)}\;dt,
\end{split}
\end{equation*}
Now, we estimate the first term in the right hand side. Applying Holder inequality and thanks to \eqref{stimaloga}, with $a=\max\left\{\frac{b N}{2(N+2)},1\right\}$, yields that
{\footnotesize
\begin{equation*}
\begin{aligned}
&\int_{\mathcal{Q}_{\tau}}|\un|^2\log^2(e+|\un|)|E|^2\;dxdt\\&\leq \left(\int_0^{\tau}\int_\Omega \vert u_n \vert^{\frac{2(N+2)}{N}} \;dxdt\right)^{\frac{N}{N+2}}\left(\int_{\mathcal{Q}_{\tau}} \left\vert E \log(e+\vert u_n\vert)\right\vert^{N+2}\;dxdt\right)^{\frac{2}{N+2}},\\&
\leq \left(\int_0^{\tau}\int_\Omega \vert u_n \vert^{\frac{2(N+2)}{N}} \;dxdt\right)^{\frac{N}{N+2}}\left(\int_{\mathcal{Q}_{\tau}} \left\vert E\right\vert^{r}\;dxdt\right)^{\frac{2}{r}}\left(\int_0^{\tau}\int_\Omega \left\vert \log(e+\vert u_n\vert)\right\vert^{\frac{r(N+2)}{r-N-2}}\;dxdt\right)^{\frac{2(r-N-2)}{r(N+2)}},\\&
\le \left(\int_0^{\tau}\int_\Omega \vert u_n \vert^{\frac{2(N+2)}{N}} \;dxdt\right)^{\frac{N}{N+2}} \Vert E\Vert^2_{L^r(\mathcal{Q}_{\tau})}\Vert \log(e+\vert u_n\vert)\Vert^2_{L^b(\mathcal{Q})},\\&\le
C\left(\int_0^{\tau}\int_\Omega \vert u_n \vert^{\frac{2(N+2)}{N}} \;dxdt\right)^{\frac{N}{N+2}} \Vert E\Vert^2_{L^r(\mathcal{Q}_{\tau})},
\end{aligned}
\end{equation*}}
where $\ds\frac1b:=\frac{1}{N+2}-\frac1r$.

 Using  Gagliardo-Nirenberg inequality together with Young inequality yields that
 \begin{equation}\label{inqlog}
\begin{aligned}
&\int_{\mathcal{Q}_{\tau}}|\un|^2\log^2(e+|\un|)|E|^2\;dxdt\\&\leq \left[c_1\sup_{\tau\in (0,\tau)}\int_\Omega \vert u_n(\tau)\vert^2 \;dx+c_2\int_0^{\tau} \|\nabla \un\|^2_{(L^2(\Omega))^N}\;dxdt \right]\Vert E\Vert^2_{L^r(\mathcal{Q}_{\tau})}
\end{aligned}
\end{equation}


Therefore, applying  Gagliardo-Nirenberg inequality, we obtain

\begin{equation*}
\begin{aligned}
&\left[\frac{1}{2}-c_1\Vert E\Vert^2_{L^r(\mathcal{Q}_{\tau})} \right]\sup_{\tau\in (0,\tau)}\int_\Omega \vert u_n(\tau)\vert^2 \;dx\\&+\left[\frac{\alpha}{4}-\frac{Cc_2}{2\alpha}\Vert E\Vert^2_{L^r(\mathcal{Q}_{\tau})}\right]\int_0^{\tau} \|\nabla \un\|^2_{(L^2(\Omega))^N}\;dt\\&\le 
\frac{k^2\log^2(e+k)}{2\alpha}\|E\|^2_{\elle2}+C\int_0^{\tau} \Vert f\Vert^2_{W^{-1,2}(\Omega)}\;dt,
\end{aligned}
\end{equation*}
where $$c_1:=\frac{2 C\bar{C}^{\frac{N}{N+2}}}{N+2} \text{ and }c_2:=\frac{N C \bar{C}^{\frac{N}{N+2}}}{N+2}.$$
Now, due to Lemma \ref{decay} we select $k$  such that 
\begin{equation*}
\left[\frac{1}{2}-c_1\Vert E\Vert^2_{L^r(\mathcal{Q}_{\tau})} \right]>0\text{ and }\left[\frac{\alpha}{4}-c_2\Vert E\Vert^2_{L^r(\mathcal{Q}_{\tau})}\right]>0,
\end{equation*}
then,  it follows that
\begin{equation*}
\begin{aligned}
&\sup_{\tau\in (0,\tau)}\int_\Omega \vert u_n(\tau)\vert^2 \;dx+\int_0^{\tau} \|\nabla \un\|^2_{(L^2(\Omega))^N}\;dt\\&\le 
\frac{k^2\log^2(e+k)}{2c_3\alpha}\|E\|^2_{\elle2}+\frac{C}{c_3}\int_0^{\tau} \Vert f\Vert^2_{W^{-1,2}(\Omega)}\;dt,
\end{aligned}
\end{equation*}
where $$c_3:=\min\left\lbrace\left[\frac{1}{2}-c_1\Vert E\Vert^2_{L^r(\mathcal{Q}_{\tau})} \right],\left[\frac{\alpha}{4}-c_2\Vert E\Vert^2_{L^r(\mathcal{Q}_{\tau})}\right]\right\rbrace.$$
Thus, there exist $\tau\in (0,T)$, independent on $n$, and a constant $M$ such that
\begin{equation}\label{esuf1}
\Vert u_n\Vert_{L^\infty(0,\tau, L^2(\Omega))}+\Vert u_n\Vert_{L^2(0,\tau, W^{1,2}_0(\Omega))}\leq \mathcal{C}.
\end{equation}

Gathering the last inequality and \eqref{inqlog}(with $\tau=T$ and $k=0$), we get
\begin{equation*}\label{esuf2}
\begin{aligned}
&\int_{\mathcal{Q}}|\un|^2\log^2(e+|\un|)|E|^2\;dxdt\leq \max\{c_1,c_2\}\mathcal{C}\Vert E\Vert^2_{L^r(\mathcal{Q}_{\tau})},
\end{aligned}
\end{equation*}
which in turn, by following the proof of Theorem \ref{generalh}, implies that the sequence $\{h(u_n)E_n(x,t)\nabla \phi\}$ is equi-integrable. 

As a consequence of \eqref{esuf1} together with Aubin's lemma, it
is possible to state the following convergences
\begin{align*}
u_n \rightharpoonup u &\quad \text{ weakly in $L^2(0,T,W^{1,2}_0(\Omega))$},\\
u_n \rightarrow u &\quad \text{ a.e. in $\mathcal{Q}$}.
\end{align*}
 Now, we can apply Vitali's theorem in order to pass to the limit as $n$ diverges and to get
\begin{equation*}
\lim_{n\rightarrow \infty} \ds\int_0^T\int_\Omega h(u_n)E_n(x,t)\nabla \varphi\;dxdt=\int_0^T\int_\Omega h(u)E(x,t)\nabla \varphi\;dxdt,
\end{equation*}
and so, by passing to the limit in the other terms in \eqref{wf}, we conclude that $u$ is a weak solution to Problem \ref{problem}.

Regarding uniqueness, observe that 
$$|h'(s)| \leq 2 + |s|, \text{ and }  |u|^2|E|^2 \in L^1(\mathcal{Q}).$$
 Consequently, Theorem \ref{comparison} guarantees uniqueness.
\end{proof}

\subsection*{Acknowledgment}
The author is member of the Gruppo Nazionale per l’Analisi Matematica, la
Probabilit\'{a} e le loro Applicazioni (GNAMPA) of the Istituto Nazionale di Alta Matem-
atica (INdAM).

\end{document}